\documentclass[reqno]{asl}

\usepackage{setspace}
\AtEndEnvironment{Example}{\hfill $\triangle$}

\title[CPL = IPL + Duality]
{Classical Logic as Intuitionistic Logic with Duality}

\author{Alexander V. Gheorghiu }
\revauthor{Gheorghiu, Alexander V.}

\address{\textbf{ORCID:} 0000-0002-7144-6910}

\address{School of Electronics and Computer Science, University of Southampton\\
University Road, Southampton, SO17 1BJ 
United Kingdom}

\address{Department of Computer Science, University College London\\
Gower St, London WC1E 6BT, UK}

\email{a.v.gheorghiu@soton.ac.uk}

\author{Yll Buzoku}
\revauthor{Buzoku, Yll}

\address{\textbf{ORCID:} 0009-0006-9478-6009}

\address{Department of Computer Science, University College London\\
Gower St, London WC1E 6BT, UK}
\email{yll.buzoku.14@ucl.ac.uk}

\thanks{\emph{Thanks}. We would like to express our gratitude to Arghya Chakrabarty, Timo Eckhardt, Tao Gu, Victor Nascimento, Elaine Pimentel, and  David J. Pym for the many discussions that led to the development of this work.}

\newtheorem{Theorem}{Theorem}[section]
\newtheorem{Proposition}[Theorem]{Proposition}
\newtheorem{Lemma}[Theorem]{Lemma}

\theoremstyle{Definition}
\newtheorem{Definition}[Theorem]{Definition}

\newtheorem{Example}[Theorem]{Example}

\usepackage{graphicx} 
\usepackage{proof}
\usepackage{enumitem}
\usepackage{mathrsfs}
\usepackage{xcolor}
\usepackage{cmll}
\usepackage{hyperref}
\hypersetup{
    colorlinks,
    linkcolor={black},
    urlcolor={black}
}
\usepackage{mathbbol}

\renewcommand{\phi}{\varphi}
\renewcommand{\emptyset}{\varnothing}

\renewcommand{\atop}{\top}

\newcommand{\prop}{\bot}

\newcommand{\dual}[1]{#1^\prop}
\newcommand{\setLiterals}{\mathbb{L}}

\newcommand{\rn}[1]{\mathsf{#1}}
\newcommand{\ern}[1]{\rn{#1}_\mathsf{E}}
\newcommand{\irn}[1]{\rn{#1}_\mathsf{I}}

\newcommand{\exc}{\rn{EXC}}
\newcommand{\dil}{\rn{DIL}}

\newcommand{\calculus}[1]{\mathsf{#1}}
\newcommand{\system}[1]{\textsc{#1}}

\newcommand{\base}[1]{\mathscr{#1}}
\newcommand{\baseB}{\base{B}}
\newcommand{\baseC}{\base{C}}
\newcommand{\baseX}{\base{X}}

\newcommand{\baseGeq}{\supseteq}

\newcommand{\supp}{\Vdash}

\newcommand{\proves}[1]{\vdash_{#1}}

\DeclareMathSymbol{\fatcomma}{\mathrel}{bbold}{\lq\,}

\newcommand{\kflat}[1]{#1^\flat}

\newcommand{\ksharp}[1]{#1^\natural}

\makeatletter
\def\labelandtag#1#2{\begingroup
   \def\@currentlabel{#2}%
   \phantomsection\label{#1}\endgroup
}
\makeatother

\begin{document}

\begin{abstract}
The field of \emph{proof-theoretic semantics} (P-tS) offers an alternative approach to meaning in logic that is based on inference and argument (rather than truth in a model). It has been successfully developed for various logics; in particular, Sandqvist has developed such semantics for both classical and intuitionistic logic. In the case of classical logic, P-tS provides a conception of consequence that avoids an \emph{a priori} commitment to the principle of bivalence, addressing what Dummett identified as a significant foundational challenge in logic.  In this paper, we propose an alternative P-tS for classical logic, which essentially extends the P-tS for intuitionistic logic by operating over \emph{literals} --- atoms signed as assertions or denials --- rather than over unsigned atoms. Importantly, literals are atomic and not defined by negation but are related by a primitive duality encoded inferentially at the atomic level. This semantics illustrates the perspective that classical logic can be understood as intuitionistic logic supplemented by a principle of duality, offering fresh insights into the relationship between these two systems.
\end{abstract}

\maketitle


\section{Introduction} \label{sec:introduction}
In this paper, we give an account of classical logic in which it is regarded as intuitionistic logic together with a duality. As a slogan:
\[
\textbf{Classical Logic} \;=\; \textbf{Intuitionistic Logic} \;+\; \textbf{Duality}
\]
Classical and intuitionistic logic need no introduction; the work of the paper is to say what `duality' means. It is a form of bilateralism: a `positive' and a `negative' version of each proposition are taken as co-equal, neither defined in terms of the other. Making this precise takes some care, philosophical as well as mathematical.

For a \emph{realist}, propositions have objective truth values independently of our capacity to establish them, and those values are typically taken to form a Boolean algebra. Duality is then trivial: it is Boolean negation, and to deny a proposition is just to assert its negation. But this account cannot serve our purposes, since it presupposes classical logic from the outset.  

For an \emph{anti-realist}, or \emph{constructivist}, matters are more delicate. On Dummett's account~\cite{Dummett1976}, anti-realism rejects verification-transcendent truth: meaning must be grounded in our actual practices of justification and verification. From this standpoint classical logic is hard to understand at all. Dummett~\cite{Dummett1978} writes:
\begin{quotation}
    `In the resolution of the conflict between [accepting classical inferences without recourse to verification-transcendent truth conditions] lies, as I see it, one of the most fundamental and intractable problems in the theory of meaning; indeed, in all philosophy.'
\end{quotation}

The anti-realist programme can be carried forward in several ways; the best known is \emph{intuitionism}. Its Brouwer--Heyting--Kolmogorov (BHK) interpretation~\cite{troelstra2014constructivism,MartinLof1985} explains the logical constants in terms of proofs: a proof of $\phi \land \psi$ is a pair consisting of a proof of $\phi$ and a proof of $\psi$; a proof of $\phi \to \psi$ is a construction transforming any proof of $\phi$ into a proof of $\psi$; and so forth. But the BHK interpretation is an informal explanatory scheme, not a formal definition.

\emph{Proof-theoretic semantics} (P-tS)~\cite{SEP-PtS,francez2015proof,wansing2000idea} makes the anti-realist metatheory rigorous. As Schroeder-Heister~\cite{Schroeder2007modelvsproof} observes, it generalises and systematises the BHK interpretation: the meaning of a logical structure is given in terms of \emph{proofs}, understood as objects denoting collections of acceptable inferences from accepted premisses. P-tS is not the same thing as a proof system. What is fixed at the outset is a notion of inference, not a formal calculus, and the semantics is defined in terms of that notion. A calculus such as Gentzen's $\mathsf{NJ}$ then stands to this semantics exactly as it stands to a model-theoretic one: it may or may not be sound and complete with respect to it, and showing that it is remains a substantive task. Soundness and completeness are, as usual, relationships between a calculus and a semantics; what changes is only that the semantics is inferential rather than denotational.

To illustrate the shift from traditional semantics to P-tS --- from denotationalism to inferentialism --- consider the proposition `Tammy is a vixen'. On an inferentialist account, its meaning is given by the rules:
\[
\infer{\text{Tammy is a vixen}}{\text{Tammy is a fox} & \text{Tammy is female}}
\qquad
\infer{\text{Tammy is female}}{\text{Tammy is a vixen}}
\qquad
\infer{\text{Tammy is a fox}}{\text{Tammy is a vixen}}
\]
These rules merit direct comparison with the natural deduction rules governing conjunction~($\land$):
\[
\frac{\phi \qquad \psi}{\phi \land \psi}
\qquad
\frac{\phi \land \psi}{\phi}
\qquad
\frac{\phi \land \psi}{\psi}
\]
That is why `Tammy is a vixen' behaves as the conjunction of `Tammy is female' and `Tammy is a fox': the rules that give it its meaning have the same shape as the rules for $\land$.

P-tS comprises several research programmes. One is \emph{proof-theoretic validity} (P-tV), in the Dummett--Prawitz tradition surveyed by Schroeder-Heister~\cite{Schroeder2006validity}. This paper works within another, \emph{base-extension semantics} (B-eS), developed by Piecha et al.~\cite{Piecha2015failure,Piecha2016completeness,Piecha2019incompleteness} and Sandqvist~\cite{Sandqvist2015hypothesis,Sandqvist2009CL,Sandqvist2015IL,Sandqvist2005inferentialist}; the two are deeply connected, as Gheorghiu and Pym~\cite{PTV-BES} have shown.

B-eS begins with an \emph{atomic system}: a collection of inferential relationships between atomic propositions, representing the beliefs an agent may hold about how thoughts are inferentially connected. The rules for Tammy's vixenhood above are one such system. Several notions of atomic rule and atomic system appear in the literature; Piecha and Schroeder-Heister~\cite{Piecha2017definitional,Schroeder2016atomic} and Sandqvist~\cite{Sandqvist2015hypothesis} analyse them systematically, building on earlier work by Prawitz~\cite{Prawitz2006natural} and Schroeder-Heister~\cite{schroeder1984natural}. Once a notion of atomic system is fixed, we call such systems \emph{bases} and write them $\base{B}$.

A base determines a notion of derivability in it ($\vdash_{\base{B}}$). Relative to this, a B-eS defines a judgement of \emph{support} ($\Vdash_{\base{B}}$) by induction on formulae, with the base case --- support for atoms --- given by derivability in the base. Validity is then support in every base:
\[
\Gamma \Vdash \phi \qquad \text{iff} \qquad \Gamma \Vdash_{\base{B}} \phi \ \text{ for any base } \base{B}.
\]
Soundness and completeness of a consequence relation $(\vdash)$ with respect to this validity are understood in the usual terms:
\begin{itemize}
    \item Soundness: $\Gamma \proves{} \phi$ implies $\Gamma \supp \phi$.
    \item Completeness: $\Gamma \supp \phi$ implies $\Gamma \proves{} \phi$.
\end{itemize}

Sandqvist~\cite{Sandqvist2009CL} gave a B-eS for classical logic in `Classical Logic without Bivalence' (\emph{Analysis}, 2009). Soundness is straightforward; completeness is harder, and draws on ideas from constructive mathematics, including bar induction. The metalogical reasoning is entirely constructive, so the result offers a genuinely anti-realist foundation for classical logic --- arguably, an answer to Dummett's challenge.

The approach has two limitations, identified by Makinson~\cite{makinson2014inferential} following de~Campos Sanz et al.~\cite{decampossanz2014constructive}. First, it covers only the connectives $\to$ and $\bot$ (and $\forall$); $\land$ and $\lor$ were added only later, by Stafford et al.~\cite{SPS2025}. Second, the completeness proof is technically involved --- so much so that Makinson~\cite{makinson2014inferential} gave an alternative proof that passes through classical model-theoretic semantics, and so departs from the anti-realist programme.

Sandqvist's B-eS for intuitionistic logic~\cite{Sandqvist2015IL} is a striking contrast. It treats the full standard language, and its completeness proof is elementary; as Sandqvist~\cite{Sandqvist2015IL} himself remarks:
\begin{quote}
    `The mathematical resources required for the purpose are quite elementary; there will be no need to invoke canonical models, K\"{o}nig's lemma, or even bar induction.'
\end{quote}
The proof works by simulating natural deduction derivations inside specially constructed bases, and it is \emph{constructive}, as Dummett's programme requires. Gheorghiu et al.~\cite{IMLL,gheorghiu2024note} have argued that the same method yields B-eS completeness for any logic whose natural deduction system has a certain symmetry, since the construction can essentially be copied.

Can the intuitionistic method be made to work for classical logic? Gheorghiu~\cite{Gheorghiu2024fol} has shown that it can, in part: it yields a completeness proof for classical B-eS, but one that does not --- and, on its own terms, cannot --- handle the full propositional language, since the disjunctive signs $\lor$ and $\exists$ fall outside its scope. How to treat the full language of classical logic by constructive means has therefore remained open. That is the problem this paper addresses, and the perspective of classical logic as intuitionistic logic plus duality is the key to it.

To make the perspective precise we need an account of \emph{duality} that an anti-realist can accept. Ours departs from the traditional reading of logical syntax in one respect: formulae are not treated as propositions in their own right but as \emph{logico-grammatical structures} of propositions --- structured components of reasoning. The view has a precedent in Wittgenstein's remark~\cite{wittgenstein2023tractatus}:
\begin{quotation}
    `5.4611 Signs for logical operations are punctuation marks.'
\end{quotation}
On this account a proposition is an atomic \emph{speech act}: a \emph{content} put forward with a certain \emph{force}. (We use `proposition' in an older sense than is now usual; Section~\ref{sec:speech-acts} explains the choice.) Speech act theory has held, since Stenius~\cite{Stenius1967}, that force and content are categorically distinct: force is to content as a functional group is to a radical in chemistry. Just as the grouping of atoms is not itself another atom, putting a content forward with a particular force is not itself a further component of that content~\cite{sep-speech-acts}. For classical logic, the forces that matter are \emph{assertion} and \emph{denial}.

The distinction is what shows why, for an anti-realist, denial cannot be reduced to assertion. An assertion commits the speaker to defending the content against the challenge `How do you know that?'; a denial commits the speaker to meeting the dual challenge. These are different normative commitments, and applying negation to the content does not turn one into the other. As Green~\cite{Green2000} observes, force concerns \emph{how} a content is meant, not \emph{what} is meant --- it is, in the terminology of speech act theory, illocutionary rather than locutionary.

Once the logical signs are read this way, duality becomes natural. Compare the traditional view, on which formulae are propositions and denial is the \emph{assertion of a negation} --- a position traceable to Frege~\cite{frege1991posthumous}:
\begin{quotation}
    `To each thought, there corresponds an opposite, so that rejecting one of them coincides with accepting the other. To make a judgment is to make a choice between opposite thoughts. Accepting one of them and rejecting the other is one act. So there is no need for a special sign for rejecting a thought. We only need a special sign for negation as such.'
\end{quotation}
If formulae are not propositions but logico-grammatical structures, there is no such thing as the `assertion of a negation', and denial cannot be reduced to assertion. The logical syntax is instead built over dual propositions, the \emph{literals} $l$ and $\dual{l}$, and the bilateralist distinction is confined to the \emph{atomic} level: only literals bear illocutionary force, while the complex formulae built from them are structures of reasoning rather than propositions in their own right.

It remains coherent to speak of the `negation of an assertion': $\dual{l}$ is \emph{logically equivalent} to $\neg l$, where $\neg$ is the sign for negation --- just as $\phi \to \psi$ is logically equivalent to $\neg\phi \lor \psi$ even when implication and disjunction are both primitive.

By way of illustration: consider an Englishman and a German attending a concert, and let $c$ be the content \emph{that the concert is good}. Asked whether the concert is good, the German is likely to assert outright --- `Yes, it is!'~($c^+$) --- or to deny outright --- `No, it is not!'~($c^-$). Note that the denial $c^-$ is not the assertion of some contrary content such as `the concert is bad': one may deny that the concert is good while declining to assert that it is bad. The Englishman, characteristically less direct, might instead employ a litotes, affirming by negating the denial: if they enjoyed the concert, they hold that it is \emph{not} the case that it is not good ($\neg c^-$); if they did not, they hold that it is \emph{not} the case that it is good ($\neg c^+$). These formulations differ in form --- indeed, in force --- but are, in each case, logically equivalent.

This is \emph{bilateralism} in the sense of Smiley~\cite{Smiley1996} and Rumfitt~\cite{rumfitt2000yes}, but restricted to atoms rather than applied uniformly at every level of logical complexity. The restriction is what makes a constructive completeness proof for the full propositional language possible. Classical logic keeps the entire proof-theoretic structure of intuitionistic logic and adds the symmetry of duality at the atoms alone, so Sandqvist's elementary completeness method for intuitionistic logic carries over, and the distinctively classical features are handled where they arise. Concretely, duality is encoded by two inferential principles on literals: \emph{exclusion} (asserting and denying the same content is absurd) and \emph{case analysis} (whatever follows both from an assertion and from the dual denial follows outright). By a result of Negri and von~Plato~\cite{Negri2008}, this pair suffices to recover classical logic from an intuitionistic base.

\medskip
\noindent\textbf{Roadmap.}
Section~\ref{sec:speech-acts} lays the philosophical foundations, separating contents, forces, propositions, and formulae, and argues for confining force to the atoms. Section~\ref{sec:cpl} gives the natural deduction system $\calculus{NK}^{\pm}$: intuitionistic logic plus two atomic rules for duality. Section~\ref{sec:bes} gives the base-extension semantics, in which the classical character lives entirely in the atomic notion of derivability while the connectives keep their intuitionistic clauses. Sections~\ref{sec:soundness} and~\ref{sec:completeness} prove soundness and completeness, the latter by Sandqvist's simulation method. Section~\ref{sec:conclusion} draws out what the result suggests for a systematic proof-theoretic semantics of classical logics.

\section{Speech Acts, Propositions, and Formulae}
\label{sec:speech-acts}

Our account rests on four notions that are easily run together: \emph{contents}, \emph{forces}, \emph{propositions}, and \emph{formulae}. This section separates them; the technical development that follows depends on keeping them separate.

\subsection{Contents, Forces, and Propositions}

Begin with the distinction between the \emph{content} of an utterance and its \emph{illocutionary force}: the content is what the utterance is about, the force is how the speaker stands to that content. The distinction is old, and the vocabulary around it has shifted.

In the tradition of Bolzano, Frege, and Russell, what had once been called a \emph{proposition} --- that which figures as a premiss or conclusion of an inference --- came to be called an \emph{assertion} or \emph{judgment}, and the word `proposition' was reserved for what Frege called a \emph{thought} (\emph{Gedanke}): an abstract content, true or false independently of any act of assertion.

Our own usage is deliberately the older one, and we should say so, because the literature we draw on --- including the bilateralist literature discussed below --- does not share it. In contemporary philosophy of language (see, for example, Green~\cite{sep-speech-acts}) a \emph{proposition} is the content of a sentence, and a speech act results when that content is combined with a force; a proposition is thus a \emph{component} of a speech act, not a speech act. Throughout this paper, by contrast, a proposition is that which figures as a premiss or conclusion of an inference: something \emph{put forward}, asserted or denied. What contemporary usage calls a proposition, we call a \emph{content}. We choose the older usage because our concern is with the inferential role of the objects of assertion and denial, and it is these, not their contents, that are the atoms of our logic. A reader who prefers the contemporary usage may substitute `atomic judgement' or `signed atom' for our `proposition', and `proposition' for our `content', throughout; nothing in the technical development turns on the words.

Frege recognised only one primitive illocutionary force, assertion, and explained denial as the assertion of a negation. The standard objection is that this conflates \emph{object-language negation} ($\neg l$, a logical operation on contents) with a meta-linguistic act of repudiation. Smiley~\cite{Smiley1996} and Rumfitt~\cite{rumfitt2000yes} answered it with \emph{bilateralism}, on which assertion and denial are both primitive forces. So understood, denying `the concert is good' is not the same speech act as asserting `the concert is not good': the two may be logically equivalent, but they are syntactically and illocutionarily distinct.

We take the bilateralist insight but depart from it in one respect. Smiley, Rumfitt, and others apply bilateralism uniformly, at every level of logical complexity: a complex formula such as $\varphi \land \psi$ can itself be asserted or denied. We apply it only at the \emph{atomic} level: only literals --- atomic propositions bearing force --- can be asserted or denied, and the complex formulae built from them by the connectives are not propositions but logico-grammatical structures that organise propositional content into patterns of reasoning. The restriction, as we shall see, gives a natural account of classical logic, one that can borrow the techniques of intuitionistic logic and avoids the complications of uniform bilateralism.

\subsection{Literals as Propositions}

A \textbf{Proposition}, in our framework, is formed by combining a content with an illocutionary force. Let $\mathcal{C}$ be a denumerable set of \textbf{contents} --- abstract objects representing what propositions are about. We consider two types of atomic propositions, which we call \textbf{literals}:
\begin{itemize}
    \item \textbf{Positive literals} $c^+$: assertions of content $c \in \mathcal{C}$.
    \item \textbf{Negative literals} $c^-$: denials of content $c \in \mathcal{C}$.
\end{itemize}
Positive and negative literals are \emph{duals} of each other, and neither is defined as the negation of the other: the duality is primitive, encoded at the level of propositions rather than mediated by a logical connective. We write $\mathbb{L}$ for the set of all literals. They are the atomic speech acts on which all reasoning in our system is built.

\begin{Example}[Dual propositions]
    Consider the content $c$: \emph{that the concert is good}. The two primitive, dual propositions are:
    \begin{itemize}
        \item $c^+$: the assertion of $c$ --- `Yes, the concert is good.'
        \item $c^-$: the denial of $c$ --- `No, the concert is not good.'
    \end{itemize}
    Two points deserve emphasis. First, $c^-$ is the \emph{denial} of $c$; it is not the assertion of some contrary content such as `the concert is bad'. Assertion and denial of the same content are \emph{contradictory} --- they can be neither jointly correct nor jointly incorrect --- whereas `good' and `bad' are merely contrary, since a concert may be neither. It is this contradictory relationship, and not mere contrariety, that the rules $\exc$ and $\dil$ of Section~\ref{sec:cpl} encode. Second, neither literal is defined in terms of the other, nor in terms of negation: they are independent, dual speech acts with a common content.
\end{Example}

\subsection{Logico-Grammatical Structures}
\label{subsec:formulae}

Following Wittgenstein's remark that `signs for logical operations are punctuation marks'~\cite{wittgenstein2023tractatus}, we treat logical connectives not as forming new propositions but as \emph{structuring devices} for complex reasoning. The set of \textbf{formulae} over $\mathbb{L}$ is defined by the following inductive grammar:
\begin{itemize}
    \item every literal in $\mathbb{L}$ is a formula;
    \item $\bot$ and $\top$ are formulae;
    \item if $\phi$ and $\psi$ are formulae, then so are $\phi \land \psi$, $\phi \lor \psi$, and $\phi \to \psi$.
\end{itemize}
We write $\neg \phi$ as an abbreviation for $\phi \to \bot$, and denote the set of all formulae by $\mathbb{F}$.

So a formula $\phi \land \psi$ is not a proposition to be asserted or denied but a logico-grammatical structure, showing how two items of propositional content are combined in reasoning. The literals $c^+$ and $c^-$ are propositions --- objects of assertion, denial, and belief; a complex formula such as $(c^+ \to d^-) \land e^+$ is a structure of reasoning that organises propositions without itself being the object of any primitive speech act.

It is a standard tenet of speech act theory that force attaches to a whole sentence, or a whole content, and not to its sub-sentential components; this is one half of the Frege--Geach point, that a clause occurring within a logically complex sentence is not thereby asserted~\cite{Geach1965}. One might therefore worry that a formula such as $c^+ \to d^-$ attaches force to sub-sentential components, in violation of this tenet. The worry rests on reading the formula as a sentence, and on our account it is not one. The connectives do not form larger sentences from smaller ones; they are, in Wittgenstein's phrase, punctuation: they articulate a pattern of reasoning whose constituents are whole, force-bearing acts. Thus $c^+ \to d^-$ is not a sentence with a conditional content, some of whose parts are asserted or denied; it records an inferential structure --- roughly, that the denial of $d$ is licensed by the assertion of $c$ --- in which each literal occurs as a complete act. Force therefore attaches, as the speech act theorist requires, only to whole items: to literals, which are atomic sentences. What is non-standard is not the locus of force but the reading of the connectives, which we take to structure inference rather than to compound contents. There is a close precedent in the treatment of hypothetical judgement in Martin-L\"{o}f's type theory~\cite{MartinLof1985}, where the hypothetical `$B$ true, assuming $A$ true' has judgements as constituents without itself being a further act of assertion.

One might object: if $c^+$ is a proposition that can be asserted, why not $c^+ \land d^+$ too? The traditional answer is that it can, but the answer brings complications. First, the conjunction $c^+ \land d^-$, which combines an assertion of $c$ with a denial of $d$, is ill-formed as a single speech act: what is its force? Second, whatever one answers, the behaviour of its putative dual $(c^+ \land d^-)^\perp = c^- \lor d^+$ is unclear: is $c^- \lor d^+$ a denial, and if so, how is it to be distinguished from an assertion of `denying $c$' or `asserting $d$'? Uniform bilateralism risks blurring the very distinction between force and content that motivated it.

These difficulties can be addressed within a fully bilateral framework. Smiley~\cite{Smiley1996} and Rumfitt~\cite{rumfitt2000yes}, and more recently Incurvati and Schl\"{o}der~\cite{IncurvatiSchloder2017,IncurvatiSchloder2019}, work with \emph{signed} formulae: every formula, of whatever complexity, is prefixed by a sign ($+$ or $-$) marking assertion or denial, and natural deduction rules are given for the signed formulae, together with coordination principles --- such as Rumfitt's \emph{Smileian reductio} --- relating the two signs. Restall~\cite{Restall2005,Restall2023}, by contrast, reads the multiple-conclusion sequent $\Gamma \Vdash \Delta$ bilaterally, as the claim that asserting everything in $\Gamma$ while denying everything in $\Delta$ is out of bounds, and develops a natural deduction counterpart with `alternatives'. For the purposes of this paper, however, we take a different route. We draw a three-way distinction:
\begin{itemize}
    \item \textbf{Propositions}: literals in $\mathbb{L}$, that is, contents paired with forces (primitive speech acts).
    \item \textbf{Formulae}: logical structures in $\mathbb{F}$ built from literals by connectives.
    \item \textbf{Judgments}: given a formula $\phi$, we write $\vdash \phi$ to express that $\phi$ is judged valid (defined in Section~\ref{sec:cpl}).
\end{itemize}
On this approach, $\Gamma \vdash \phi \land \psi$ means: from the premisses $\Gamma$, one can derive the structured conclusion that both $\phi$ and $\psi$ hold. The formula $\phi \land \psi$ is not itself asserted; it represents the logical structure of a valid inference pattern.

Why prefer this to uniform bilateralism? Three considerations weigh in its favour. The first is \emph{economy}. On the uniform view, the two forces must be coordinated at every level of logical complexity: one needs rules saying what it is to deny a conjunction, a conditional, and so on, together with principles coordinating the signs. On our view, the coordination is stated exactly once, at the atomic level, by the rules $\exc$ and $\dil$ of Section~\ref{sec:cpl}; the connectives are untouched. The second is \emph{modularity}. Because the connectives carry precisely their intuitionistic meaning, what is known about the proof-theoretic semantics of intuitionistic logic --- in particular, Sandqvist's elementary completeness method --- transfers directly. This is what makes the constructive completeness proof of Section~\ref{sec:completeness} possible; uniform bilateralism affords no such transfer, since it revises the meaning of every connective. The third is \emph{fidelity to the force--content distinction}. The uniform view must say what the force of a signed complex formula is, and here the categorical distinction with which bilateralism begins comes under strain: the Frege--Geach point, discussed above, is precisely that force does not embed under the connectives. Restricting force to literals respects the point rather than legislating around it.

None of this shows the uniform approach to be mistaken --- it has its own well-developed motivations, notably a direct account of multiple-conclusion consequence --- but it does show that the atomic approach is not merely a technical expedient. It is the natural home for the thesis that classical logic is intuitionistic logic with duality.

\begin{Example}
\label{ex:concert}
    Consider two speakers discussing a concert. The proposition $c^+$ --- the speech act of asserting that the concert is good --- is distinct from the judgment $\vdash \neg c^-$, which records the validity of a logical inference. A defining feature of classical logic is that the judgment $\vdash c^+$ and the judgment $\vdash \neg c^-$ are equivalent, expressed as $\vdash c^+ \leftrightarrow \neg c^-$.
\end{Example}

To develop the technical side of this account, we introduce the \emph{duality operator} $\dual{(-)}$, already suggested by the examples above. It is defined on literals, and on literals only:
\begin{align*}
    \dual{(c^+)} &= c^- & &\text{(the dual of asserting $c$ is denying $c$),} \\
    \dual{(c^-)} &= c^+ & &\text{(the dual of denying $c$ is asserting $c$).}
\end{align*}
Note that $\dual{(-)}$ is an involution on $\setLiterals$. We deliberately do \emph{not} extend it to complex formulae: in keeping with the view that only literals bear force, duality is a relationship between speech acts, not between logico-grammatical structures. Where a `dual' of a complex formula is wanted, its role is played by negation, $\neg\phi := \phi \to \bot$, and the relationship between the two is settled inferentially (Proposition~\ref{prop:negation}).

\section{Classical Propositional Logic}
\label{sec:cpl}

The setup of Section~\ref{sec:speech-acts} is built to deliver the slogan of the paper:
\begin{center}
    \textbf{Classical Logic = Intuitionistic Logic + Duality.}
\end{center}
Intuitionistic logic can be developed with positive propositions --- assertions --- alone. Classical logic emerges when negative propositions --- denials --- are added as primitives, together with a duality relating each literal to its dual.

What makes the duality \emph{classical}? A realist would say: $l$ and $\dual{l}$ carry opposite truth values --- not necessarily in the two-valued sense, but in the sense that the value of $\dual{l}$ is the Boolean complement of the value of $l$ in whatever algebra of truth values one adopts, as in Boolean-valued model theory~\cite{bell2011set}. The anti-realist has no truth values to appeal to, and must express the duality entirely through inferential behaviour.

There is a well-known obstacle to reading duality off the rules for negation. Carnap~\cite{carnap1943} observed that the standard single-conclusion formalisations of classical logic fail to exclude non-normal interpretations of the connectives: the rules are satisfied, for instance, by a valuation making every sentence true, and by valuations making both $\phi$ and $\neg\phi$ false. The observation has been taken up in various ways --- by Belnap~\cite{Belnap1962}, Smiley~\cite{Smiley1996}, Rumfitt~\cite{rumfitt2000yes}, Garson~\cite{Garson2013}, and Raatikainen~\cite{raatikainen2008}, among others --- and one lesson commonly drawn from it is that the rules of a single-conclusion natural deduction system do not, by themselves, fix negation as a contradictory-forming operator.

The relevance of this to the present paper is narrow but real: it means that the duality of assertion and denial cannot simply be identified with, or read off from, the behaviour of $\neg$ in $\mathsf{NJ}$. Bilateralists such as Smiley and Rumfitt respond by taking the coordination of assertion and denial as primitive; we do the same, with the difference that we do so only at the atomic level. Concretely, we require inferential principles, stated for literals, that make a literal and its dual jointly \emph{exclusive} and jointly \emph{exhaustive}.

Negri and von Plato~\cite{Negri2008} show that such principles suffice: adding decidability of atomic formulae as a structural assumption closes the gap between intuitionistic and classical logic. Under this reading, classical logic is the system in which every atomic proposition $p$ satisfies $p \lor \neg p$, expressed inferentially by the rule
\[
    \infer{\phi}{\deduce{\phi}{[p]} & \deduce{\phi}{[\neg p]}.}
\]
Adding this rule to Gentzen's $\system{NJ}$ recovers classical logic. We take this as the inferentialist template for duality.

\begin{Definition}[Natural deduction system $\mathsf{NK}^\pm$]
\label{def:nk-pm}
    The system $\mathsf{NK}^\pm$ consists of the rules in Figure~\ref{fig:nk-pm}.
\end{Definition}

\begin{figure}[t]
    \hrule\vspace{2mm}
    \[
    \begin{array}{c}
        \infer[\top\mathsf{I}]{~~\top~~}{}
        \qquad
        \infer[\bot\mathsf{E}]{\varphi}{\bot}
        \\[3mm]
        \infer[{\to}\mathsf{I}]{\varphi \to \psi}{\deduce{\psi}{[\varphi]}}
        \qquad
        \infer[{\to}\mathsf{E}]{\psi}{\varphi \to \psi & \varphi}
        \\[3mm]
        \infer[{\land}\mathsf{I}]{\varphi \land \psi}{\varphi & \psi}
        \qquad
        \infer[{\land}\mathsf{E}^1]{\varphi_1}{\varphi_1 \land \varphi_2}
        \quad
        \infer[{\land}\mathsf{E}^2]{\varphi_2}{\varphi_1 \land \varphi_2}
        \\[3mm]
        \infer[{\lor}\mathsf{I}^1]{\varphi_1 \lor \varphi_2}{\varphi_1}
        \quad
        \infer[{\lor}\mathsf{I}^2]{\varphi_1 \lor \varphi_2}{\varphi_2}
        \qquad
        \infer[{\lor}\mathsf{E}]{\psi}{\varphi_1 \lor \varphi_2 & \deduce{\psi}{[\varphi_1]} & \deduce{\psi}{[\varphi_2]}}
        \\[3mm]
        \dotfill
        \\[3mm]
        \infer[\exc]{\varphi}{~~l & \dual{l}}
        \qquad
        \infer[\dil]{\varphi}{\deduce{\varphi}{[l]} & \deduce{\varphi}{[\dual{l}]}}
    \end{array}
    \]
    \vspace{1mm}\hrule
    \caption{Natural deduction system $\mathsf{NK}^\pm$.}
    \label{fig:nk-pm}
\end{figure}

The rules above the dotted line constitute a standard natural deduction system for \emph{intuitionistic propositional logic}, traditionally denoted $\mathsf{NJ}$. The rules below govern the duality operator and are specific to classical logic:
\begin{itemize}
    \item $\exc$ (\emph{exclusion}): from $l$ and $\dual{l}$ one derives $\varphi$.
    \item $\dil$ (\emph{case analysis}): if $\varphi$ is derivable from $l$ and also from $\dual{l}$, then $\varphi$ is derivable outright.
\end{itemize}
In both rules, $l$ is a literal and $\dual{l}$ is its dual; $\varphi$ may be any formula. We write $\Gamma \vdash_{\system{NK}^\pm} \phi$, or simply $\Gamma \vdash \phi$, to mean that $\phi$ is derivable from $\Gamma$ in $\system{NK}^\pm$.

Together, the two rules ensure that dual literals behave as negations of each other: $\exc$ delivers one direction of the equivalence, and $\dil$ the other.

\begin{Proposition}\label{prop:negation}
    For every literal $l$, $\vdash_{\system{NK}^\pm} \dual{l} \leftrightarrow \neg l$.
\end{Proposition}

\begin{proof}
    The two directions are witnessed by the following derivations, in which square brackets mark discharged assumptions:
    \[
        \infer[{\to}\mathsf{I}]{\dual{l} \to \neg l}{
            \infer[{\to}\mathsf{I}]{\neg l}{
                \infer[\exc]{\bot}{[l] & [\dual{l}]}}}
        \qquad\qquad
        \infer[{\to}\mathsf{I}]{\neg l \to \dual{l}}{
            \infer[\dil]{\dual{l}}{
                \infer[\bot\mathsf{E}]{\dual{l}}{
                    \infer[{\to}\mathsf{E}]{\bot}{[\neg l] & [l]}}
                & [\dual{l}]}}
    \]
    On the left, $\exc$ yields $\bot$ from $l$ and $\dual{l}$, and two applications of ${\to}\mathsf{I}$ discharge the assumptions. On the right, $\dil$ is applied with case literal $l$: in the first case, from $l$ and the open assumption $\neg l$ we obtain $\bot$ and hence $\dual{l}$ by $\bot\mathsf{E}$; in the second case, $\dual{l}$ is assumed outright. Both case assumptions are discharged by $\dil$, and ${\to}\mathsf{I}$ then discharges $\neg l$.
\end{proof}

Observe the division of labour: the left-to-right direction uses only $\exc$, which makes a literal and its dual jointly exclusive; the right-to-left direction requires $\dil$, which makes them jointly exhaustive. The rule $\dil$ is also what makes the resulting system classical: together with $\exc$, it recovers full classical logic from the intuitionistic base, in the sense of Negri and von Plato~\cite{Negri2008}.

\section{Base-Extension Semantics}
\label{sec:bes}

The natural deduction system $\system{NK}^\pm$ fixes a notion of derivability. It is, however, a formal system, and --- as explained in Section~\ref{sec:introduction} --- a formal system is not a semantics: the meaning of the logical constants is to be given by an account of inference in general, against which any particular calculus is then measured. That account is what we now supply.

Our metatheory is anti-realist. Since the point of the exercise is to understand classical logic as intuitionistic logic together with duality, and since intuitionistic logic is the logic of anti-realist reasoning, we require a semantics that takes inferential relationships, rather than truth values, as primitive. This rules out a model-theoretic semantics, Boolean or otherwise --- not because such semantics are inadequate to classical logic, but because they presuppose the realist conception of truth that the anti-realist declines to take for granted. What we develop instead is a proof-theoretic semantics: a base-extension semantics in the sense of Sandqvist~\cite{Sandqvist2015IL}, extended to literals. The adequacy of $\system{NK}^\pm$ with respect to it is then established, by constructive means, in the soundness and completeness theorems of Sections~\ref{sec:soundness} and~\ref{sec:completeness}.

Base-extension semantics begins with a \emph{base}: a collection of inferential relationships between atomic propositions, here literals. These relationships are pre-logical, so $\bot$ does not appear among them; it is a logical sign marking absurdity, not a proposition.

A base represents the inferential beliefs an agent may hold about connections between thoughts. Whether its rules are better understood as encoding `knowledge' or `definition' is debated; we refer the reader to the analyses of atomic systems by Piecha and Schroeder-Heister~\cite{Piecha2017definitional,Schroeder2016atomic} and Sandqvist~\cite{Sandqvist2015hypothesis}, which build on earlier work by Prawitz~\cite{Prawitz2006natural} and Schroeder-Heister~\cite{schroeder1984natural}.

\begin{Definition}[Atomic rule]\label{def:atomic-rule}
    An \emph{atomic rule} is an expression of one of the following forms:
    \begin{enumerate}[label=(\roman*)]
        \item $L \Rightarrow l$, or
        \item $(L_1 \Rightarrow l_1), \ldots, (L_n \Rightarrow l_n) \Rightarrow l$,
    \end{enumerate}
    where $l, l_1, \ldots, l_n \in \setLiterals$ and $L, L_1, \ldots, L_n \subseteq \setLiterals$ are finite, possibly empty, sets of literals.
\end{Definition}

Atomic rules are read as natural deduction rules in the sense of Gentzen~\cite{Gentzen}. A rule $L \Rightarrow l$, where $L = \{l_1, \ldots, l_n\}$, corresponds to
\[
    \infer{l}{l_1 & \cdots & l_n},
\]
and a rule $(L_1 \Rightarrow l_1), \ldots, (L_n \Rightarrow l_n) \Rightarrow l$ corresponds to
\[
    \infer{l}{\deduce{l_1}{[L_1]} & \cdots & \deduce{l_n}{[L_n]}}.
\]
We write $\Rightarrow l$ as shorthand for $\emptyset \Rightarrow l$, that is, an axiom asserting $l$ unconditionally. A \emph{base} is a set of atomic rules.

\begin{Definition}[Base]
    A \emph{base} $\base{B}$ is a, possibly infinite, set of atomic rules.
\end{Definition}

Reading atomic rules as natural deduction patterns determines a notion of \emph{derivability in a base}. To capture classical reasoning at the atomic level, we supplement the standard structural clauses with two further clauses encoding duality.

\begin{Definition}[Derivability in a base]\label{def:derivability-in-a-base}
    Let $\base{B}$ be a base. \emph{Derivability in $\base{B}$}, written $\proves{\base{B}}$, is inductively defined as follows for any $L \subseteq \setLiterals$:
    \begin{enumerate}
        \item[\textsc{ref}.] If $l \in L$, then $L \proves{\base{B}} l$.
        \item[\textsc{app}$_1$.] If $\Rightarrow l \in \base{B}$, then $L \proves{\base{B}} l$.
        \item[\textsc{app}$_2$.] If $(L_1 \Rightarrow l_1), \ldots, (L_n \Rightarrow l_n) \Rightarrow l \in \base{B}$ and $L, L_i \proves{\base{B}} l_i$ for each $i = 1, \ldots, n$, then $L \proves{\base{B}} l$.
        \item[\textsc{exc}.] If $L \proves{\base{B}} l$ and $L \proves{\base{B}} \dual{l}$, then $L \proves{\base{B}} m$ for every $m \in \setLiterals$.
        \item[\textsc{dil}.] If $l, L \proves{\base{B}} m$ and $\dual{l}, L \proves{\base{B}} m$, then $L \proves{\base{B}} m$.
    \end{enumerate}
\end{Definition}

Restricting to clauses \textsc{ref}, \textsc{app}$_1$, and \textsc{app}$_2$ recovers the standard notion of derivability in a base used by Sandqvist~\cite{Sandqvist2015IL} for intuitionistic propositional logic. The clauses \textsc{exc} and \textsc{dil} extend this to classical propositional logic by incorporating duality at the literal level --- precisely where, as established in Section~\ref{sec:speech-acts}, bilateralism operates.

\subsection{Support}

Relative to derivability in bases, the meaning of logical connectives is given by a judgement called \emph{support}, defined inductively on the structure of formulae.

\begin{Definition}[Support]\label{def:support}
    \emph{Support} ($\supp$) is inductively defined by the clauses in Figure~\ref{fig:supp-cpl}.
\end{Definition}

\begin{figure}[ht]
    \hrule\vspace{2mm}
    \[
    \begin{array}{lcl @{\hspace{4em}} r}
        \supp_{\base{B}} l & \text{iff} & \proves{\base{B}} l & \ref{cl:CPL:at} \\[2mm]
        \supp_{\base{B}} \bot & \text{iff} & \proves{\base{B}} l \text{ for every } l \in \setLiterals & \ref{cl:CPL:bot} \\[2mm]
        \supp_{\base{B}} \top & & \text{always} & \ref{cl:CPL:top} \\[2mm]
        \supp_{\base{B}} \phi \to \psi & \text{iff} & \phi \supp_{\base{B}} \psi & \ref{cl:CPL:to} \\[2mm]
        \supp_{\base{B}} \phi \land \psi & \text{iff} & \supp_{\base{B}} \phi \text{ and } \supp_{\base{B}} \psi & \ref{cl:CPL:and} \\[2mm]
        \supp_{\base{B}} \phi \lor \psi & \text{iff} & \forall\,\base{C} \supseteq \base{B},\ \forall\, l \in \setLiterals\colon \text{ if } \phi \supp_{\base{C}} l \text{ and } \psi \supp_{\base{C}} l\text{, then } \supp_{\base{C}} l & \ref{cl:CPL:or} \\[2mm]
        \Delta \supp_{\base{B}} \phi & \text{iff} & \forall\,\base{C} \supseteq \base{B}\colon \text{ if } \supp_{\base{C}} \delta \text{ for every } \delta \in \Delta\text{, then } \supp_{\base{C}} \phi & \ref{cl:CPL:inf} \\[2mm]
        \Gamma \supp \phi & \text{iff} & \Gamma \supp_{\base{B}} \phi \text{ for every base } \base{B} &
    \end{array}
    \]
    \vspace{1mm}\hrule
    \caption{Support for classical propositional logic.}
    \label{fig:supp-cpl}
\end{figure}

\labelandtag{cl:CPL:at}{\ensuremath{(\text{At})}}
\labelandtag{cl:CPL:bot}{\ensuremath{(\bot)}}
\labelandtag{cl:CPL:top}{\ensuremath{(\top)}}
\labelandtag{cl:CPL:to}{\ensuremath{(\to)}}
\labelandtag{cl:CPL:and}{\ensuremath{(\land)}}
\labelandtag{cl:CPL:or}{\ensuremath{(\lor)}}
\labelandtag{cl:CPL:inf}{\ensuremath{(\text{Inf})}}

Definition~\ref{def:support} is inductive, but the induction measure is not syntactic size. It is instead the \emph{logical weight} $w(\varphi)$, defined following Sandqvist~\cite{Sandqvist2015IL} by:
\[
    w(\varphi) :=
    \begin{cases}
        0 & \text{if } \varphi \in \setLiterals, \\
        1 & \text{if } \varphi \in \{\bot, \top\}, \\
        w(\varphi_1) + w(\varphi_2) + 1 & \text{if } \varphi = \varphi_1 \circ \varphi_2,\ \circ \in \{\to, \land, \lor\}.
    \end{cases}
\]
In each clause of Figure~\ref{fig:supp-cpl}, the total weight of complex formulae in the definiens is strictly less than that in the definiendum, so the definition is well-founded. We call induction on this measure \emph{semantic induction}, to distinguish it from structural induction on syntactic complexity.

The support clauses in Figure~\ref{fig:supp-cpl} are precisely those of Sandqvist~\cite{Sandqvist2015IL} for intuitionistic propositional logic. This is not incidental: it reflects our central thesis that classical logic is intuitionistic logic plus duality at the atomic level. The classical character of the semantics is carried entirely by the clauses \textsc{exc} and \textsc{dil} of Definition~\ref{def:derivability-in-a-base}, which operate only on literals. Above the atomic level, the connectives receive exactly their intuitionistic meaning --- in keeping with the view of complex formulae as logico-grammatical structures rather than propositions in their own right.

A natural question is whether, in the bilateral setting, the clause for disjunction simplifies to the Kripke condition
\[
    \supp_{\base{B}} \phi \lor \psi \qquad \text{iff} \qquad \supp_{\base{B}} \phi \text{ or } \supp_{\base{B}} \psi.
\]
It does not. To see this, let $\phi = l$ and $\psi = m$ be literals, and let
\[
    \base{B} := \bigl\{\,(l \Rightarrow x),\, (m \Rightarrow x) \Rightarrow x \;\big|\; x \in \setLiterals\,\bigr\}.
\]
Then $\supp_{\base{B}} l \lor m$ holds under Sandqvist's clause but fails under the Kripke condition, since neither $\supp_{\base{B}} l$ nor $\supp_{\base{B}} m$ need hold.

We claim that the semantics characterises classical propositional logic:
\[
    \Gamma \supp \varphi \qquad \text{if and only if} \qquad \Gamma \vdash_{\system{NK}^\pm} \varphi.
\]
This is established by soundness (Theorem~\ref{thm:cpl-soundness}) and completeness (Theorem~\ref{thm:CPL:completeness-support}) in the sections that follow.

\subsection{Basic Properties}

We record several elementary but useful properties of derivability in a base and of support. The first two are straightforward.

\begin{Proposition}[Weakening]\label{prop:CPL:hyp-weakening}
    For any base $\base{B}$, if $L \proves{\base{B}} l$, then $M, L \proves{\base{B}} l$.
\end{Proposition}

\begin{proof}
    By induction on the derivation of $L \proves{\base{B}} l$.

    \medskip
    \textsc{ref}: Since $l \in L \subseteq M \cup L$, the result follows by \textsc{ref}.

    \medskip
    \textsc{app}$_1$: We have $\Rightarrow l \in \base{B}$. Since $M \cup L \subseteq \setLiterals$, the result follows by \textsc{app}$_1$.

    \medskip
    \textsc{app}$_2$: There is a rule $(L_1 \Rightarrow l_1), \ldots, (L_n \Rightarrow l_n) \Rightarrow l \in \base{B}$ with $L, L_i \proves{\base{B}} l_i$ for each $i$. By the induction hypothesis, $M, L, L_i \proves{\base{B}} l_i$ for each $i$. The result follows by \textsc{app}$_2$.

    \medskip
    \textsc{exc}: There exists $n \in \setLiterals$ such that $L \proves{\base{B}} n$ and $L \proves{\base{B}} \dual{n}$. By the induction hypothesis, $M, L \proves{\base{B}} n$ and $M, L \proves{\base{B}} \dual{n}$. The result follows by \textsc{exc}.

    \medskip
    \textsc{dil}: There exists $n \in \setLiterals$ such that $L, n \proves{\base{B}} l$ and $L, \dual{n} \proves{\base{B}} l$. By the induction hypothesis, $M, L, n \proves{\base{B}} l$ and $M, L, \dual{n} \proves{\base{B}} l$. The result follows by \textsc{dil}.
\end{proof}

The following propositions are proved by the same arguments as in Sandqvist~\cite{Sandqvist2015IL}, with straightforward additional cases for \textsc{exc} and \textsc{dil}; we elide the details.

\begin{Proposition}[Monotonicity]\label{prop:CPL:base-weakening}
    If $L \proves{\base{B}} l$ and $\base{C} \supseteq \base{B}$, then $L \proves{\base{C}} l$.
\end{Proposition}

\begin{Proposition}\label{prop:emptybase}
    $\Gamma \supp \phi$ if and only if $\Gamma \supp_{\emptyset} \phi$.
\end{Proposition}

The following result is the key technical lemma underlying completeness. It is a generalisation of Lemma~2.2 of Sandqvist~\cite{Sandqvist2015IL}, extended to handle the two exclusion clauses.

\begin{Proposition}\label{prop:CPL:atomic-cut}
    For any base $\base{B}$, any $M, L \subseteq \setLiterals$, and any $l \in \setLiterals$,
    \[
        M, L \proves{\base{B}} l
        \qquad \text{iff} \qquad
        \forall\,\base{C} \supseteq \base{B}\colon \text{ if } \proves{\base{C}} m \text{ for every } m \in M\text{, then } L \proves{\base{C}} l.
    \]
\end{Proposition}

\begin{proof}
    $(\Rightarrow)$\quad Assume $M, L \proves{\base{B}} l$. Let $\base{C} \supseteq \base{B}$ be such that $\proves{\base{C}} m$ for every $m \in M$. We show $L \proves{\base{C}} l$ by induction on the derivation of $M, L \proves{\base{B}} l$.

    \medskip
    \textsc{ref}: Either $l \in L$, in which case \textsc{ref} gives $L \proves{\base{C}} l$ directly, or $l \in M$, in which case $\proves{\base{C}} l$ by hypothesis, and Proposition~\ref{prop:CPL:hyp-weakening} gives $L \proves{\base{C}} l$.

    \medskip
    \textsc{app}$_1$: We have $\Rightarrow l \in \base{B} \subseteq \base{C}$. The result follows by \textsc{app}$_1$.

    \medskip
    \textsc{app}$_2$: There is a rule $(L_1 \Rightarrow l_1), \ldots, (L_n \Rightarrow l_n) \Rightarrow l \in \base{B} \subseteq \base{C}$ with $M, L, L_i \proves{\base{B}} l_i$ for each $i$. By the induction hypothesis, $L, L_i \proves{\base{C}} l_i$ for each $i$. The result follows by \textsc{app}$_2$.

    \medskip
    \textsc{exc}: There exists $n \in \setLiterals$ such that $M, L \proves{\base{B}} n$ and $M, L \proves{\base{B}} \dual{n}$. By the induction hypothesis, $L \proves{\base{C}} n$ and $L \proves{\base{C}} \dual{n}$. The result follows by \textsc{exc}.

    \medskip
    \textsc{dil}: There exists $n \in \setLiterals$ such that $M, L, n \proves{\base{B}} l$ and $M, L, \dual{n} \proves{\base{B}} l$. By the induction hypothesis, $L, n \proves{\base{C}} l$ and $L, \dual{n} \proves{\base{C}} l$. The result follows by \textsc{dil}.

    \bigskip
    $(\Leftarrow)$\quad Set $\base{C} := \base{B} \cup \{\Rightarrow m \mid m \in M\}$. Then $\proves{\base{C}} m$ for every $m \in M$ by \textsc{app}$_1$, so the hypothesis gives $L \proves{\base{C}} l$. We show $M, L \proves{\base{B}} l$ by induction on the derivation of $L \proves{\base{C}} l$.

    \medskip
    \textsc{ref}: We have $l \in L$. The result follows by \textsc{ref}.

    \medskip
    \textsc{app}$_1$: We have $\Rightarrow l \in \base{C}$. By definition of $\base{C}$, either $\Rightarrow l \in \base{B}$, in which case \textsc{app}$_1$ gives $M, L \proves{\base{B}} l$, or $l \in M$, in which case \textsc{ref} gives $M, L \proves{\base{B}} l$.

    \medskip
    \textsc{app}$_2$: There is a rule $(L_1 \Rightarrow l_1), \ldots, (L_n \Rightarrow l_n) \Rightarrow l \in \base{C}$ with $L, L_i \proves{\base{C}} l_i$ for each $i$. Since rules of the form $(L_1 \Rightarrow l_1), \ldots \Rightarrow l$ do not appear in $\{\Rightarrow m \mid m \in M\}$, the rule lies in $\base{B}$. By the induction hypothesis, $M, L, L_i \proves{\base{B}} l_i$ for each $i$. The result follows by \textsc{app}$_2$.

    \medskip
    \textsc{exc}: There exists $n \in \setLiterals$ such that $L \proves{\base{C}} n$ and $L \proves{\base{C}} \dual{n}$. By the induction hypothesis, $M, L \proves{\base{B}} n$ and $M, L \proves{\base{B}} \dual{n}$. The result follows by \textsc{exc}.

    \medskip
    \textsc{dil}: There exists $n \in \setLiterals$ such that $L, n \proves{\base{C}} l$ and $L, \dual{n} \proves{\base{C}} l$. By the induction hypothesis, $M, L, n \proves{\base{B}} l$ and $M, L, \dual{n} \proves{\base{B}} l$. The result follows by \textsc{dil}.
\end{proof}

With Proposition~\ref{prop:CPL:atomic-cut} in hand, we can verify directly that Peirce's law is valid in this semantics:
\[
\supp ((\varphi\to\psi)\to\varphi)\to\varphi
\]
We treat only the case in which $\varphi$ and $\psi$ are atoms. The remaining cases are structurally similar to the inductive cases in the proof of soundness (Section~\ref{sec:soundness}), so we omit them.
\begin{Example}[Peirce's law]\label{ex:SemanticPeircesLaw}
    Let $p$ and $q$ be atoms. We show that $\supp ((p\to q)\to p)\to p$. By~\ref{cl:CPL:to} and~\ref{cl:CPL:inf}, it suffices to fix an arbitrary base $\baseB$ such that $\supp_{\baseB} (p\to q)\to p$ and to show that $\supp_{\baseB}p$.

    We first unpack the hypothesis. By~\ref{cl:CPL:to}, it says that $p\to q\supp_{\baseB}p$; by~\ref{cl:CPL:inf} and~\ref{cl:CPL:to} again, it says that for every base $\baseX\baseGeq\baseB$, if $p\supp_{\baseX}q$, then $\supp_{\baseX}p$. Between atoms, support coincides with derivability in a base (by~\ref{cl:CPL:inf}, \ref{cl:CPL:at}, and Proposition~\ref{prop:CPL:atomic-cut}), so the hypothesis amounts to this:
    \[
        \text{for every base } \baseX\baseGeq\baseB\text{, if } p\vdash_{\baseX}q\text{, then } \vdash_{\baseX}p.
    \]

    We next reduce the goal. By~\ref{cl:CPL:at}, showing $\supp_{\baseB}p$ means showing $\vdash_{\baseB}p$. Here the classical content of the semantics enters: by~\textsc{dil}, which splits on $p$ and its dual, $\vdash_{\baseB}p$ follows from $p\vdash_{\baseB}p$ and $p^\bot\vdash_{\baseB}p$. The first holds by~\textsc{ref}. For the second, Proposition~\ref{prop:CPL:atomic-cut} tells us that it suffices to fix an arbitrary base $\baseC\baseGeq\baseB$ such that $\vdash_{\baseC}p^\bot$ and to show that $\vdash_{\baseC}p$.

    Finally, we close the gap. By Proposition~\ref{prop:CPL:hyp-weakening}, $\vdash_{\baseC}p^\bot$ gives $p\vdash_{\baseC}p^\bot$, and \textsc{ref} gives $p\vdash_{\baseC}p$. From $p$, then, $\baseC$ derives both $p$ and its dual, so by~\textsc{exc} it derives every literal; in particular, $p\vdash_{\baseC}q$. Since $\baseC\baseGeq\baseB$, the hypothesis now yields $\vdash_{\baseC}p$, as required.
\end{Example}

\section{Soundness}
\label{sec:soundness}

\begin{Theorem}[Soundness]\label{thm:cpl-soundness}
    If $\Gamma \vdash_{\mathsf{NK}^\pm} \phi$, then $\Gamma \supp \phi$.
\end{Theorem}

\begin{proof}
    Fix a base $\base{B}$ and suppose $\supp_{\base{B}} \gamma$ for every $\gamma \in \Gamma$. It suffices to show $\supp_{\base{B}} \phi$. We proceed by induction on the $\mathsf{NK}^\pm$-derivation of $\phi$.

    The cases corresponding to the rules of $\mathsf{NJ}$ (i.e., all rules in Figure~\ref{fig:nk-pm} above the dotted line) follow by the same arguments as in Sandqvist~\cite{Sandqvist2015IL}, since those rules and the corresponding support clauses are unchanged. It remains to treat the two classical rules.

    \medskip
    \noindent$\exc$:\enspace Suppose $\supp_{\base{B}} l$ and $\supp_{\base{B}} \dual{l}$. By~\ref{cl:CPL:at}, we have $\proves{\base{B}} l$ and $\proves{\base{B}} \dual{l}$. By \textsc{exc}, therefore $\proves{\base{B}} m$ for every $m \in \setLiterals$, from whence, by a similar argument to that in in Sandqvist~\cite{Sandqvist2015IL} for~\ref{cl:CPL:bot}, follows $\supp_{\baseB}\varphi$, for any formula $\varphi$, as required.

    \medskip
    \noindent$\dil$:\enspace Suppose $l \supp_{\base{B}} \phi$ and $\dual{l} \supp_{\base{B}} \phi$ for some literal $l$. We show $\supp_{\base{B}} \phi$ by induction on $w(\phi)$ presented as a case analysis on the structure of $\phi$.

    \medskip
    \emph{$\phi \in \setLiterals$.}\enspace By~\ref{cl:CPL:inf} and~\ref{cl:CPL:at}, the hypotheses give $l \proves{\base{B}} \phi$ and $\dual{l} \proves{\base{B}} \phi$. Applying \textsc{dil} yields $\proves{\base{B}} \phi$, and hence $\supp_{\base{B}} \phi$ by~\ref{cl:CPL:at}.

    \medskip
    \emph{$\phi = \top$.}\enspace Immediate from~\ref{cl:CPL:top}.

    \medskip
    \emph{$\phi = \bot$.}\enspace It suffices to show $\proves{\base{B}} m$ for every $m \in \setLiterals$. Fix $m \in \setLiterals$. By~\ref{cl:CPL:inf} and~\ref{cl:CPL:at}, the hypotheses give $l \proves{\base{B}} m$ and $\dual{l} \proves{\base{B}} m$. Applying \textsc{dil} yields $\proves{\base{B}} m$, as required.

    \medskip
    \emph{$\phi = \alpha \land \beta$.}\enspace By~\ref{cl:CPL:and}, it suffices to show $\supp_{\base{B}} \alpha$ and $\supp_{\base{B}} \beta$ separately. For each conjunct, the hypotheses imply $l \supp_{\base{B}} \alpha$ and $\dual{l} \supp_{\base{B}} \alpha$ (respectively for $\beta$) by~\ref{cl:CPL:inf} and~\ref{cl:CPL:and}. The conclusion follows by the induction hypothesis, since $w(\alpha), w(\beta) < w(\phi)$.

    \medskip
    \emph{$\phi = \alpha \to \beta$.}\enspace Let $\base{C} \supseteq \base{B}$ and suppose $\supp_{\base{C}} \alpha$; we must show $\supp_{\base{C}} \beta$. By~\ref{cl:CPL:inf} and~\ref{cl:CPL:to}, the hypotheses give $l, \alpha \supp_{\base{B}} \beta$ and $\dual{l}, \alpha \supp_{\base{B}} \beta$. By Proposition~\ref{prop:CPL:base-weakening} and~\ref{cl:CPL:inf}, therefore $l \supp_{\base{C}} \beta$ and $\dual{l} \supp_{\base{C}} \beta$. The conclusion $\supp_{\base{C}} \beta$ follows by the induction hypothesis.

    \medskip
    \emph{$\phi = \alpha \lor \beta$.}\enspace Let $\base{C} \supseteq \base{B}$ and $m \in \setLiterals$, and suppose $\alpha \supp_{\base{C}} m$ and $\beta \supp_{\base{C}} m$; we must show $\supp_{\base{C}} m$ (cf.~\ref{cl:CPL:or}). By the argument of the base case (with $\base{C}$ in place of $\base{B}$), it suffices to show $l \supp_{\base{C}} m$ and $\dual{l} \supp_{\base{C}} m$; we treat the former, the latter being symmetric.

    By hypothesis, $l \supp_{\base{B}} \alpha \lor \beta$. By Proposition~\ref{prop:CPL:base-weakening} and~\ref{cl:CPL:inf}, this gives $l \supp_{\base{C}} \alpha \lor \beta$. Unfolding~\ref{cl:CPL:inf} and~\ref{cl:CPL:or}: for every $\base{D} \supseteq \base{C}$, if $\supp_{\base{D}} l$ and $\alpha \supp_{\base{D}} m$ and $\beta \supp_{\base{D}} m$, then $\supp_{\base{D}} m$. Since $\alpha \supp_{\base{C}} m$ and $\beta \supp_{\base{C}} m$, Proposition~\ref{prop:CPL:base-weakening} gives $\alpha \supp_{\base{D}} m$ and $\beta \supp_{\base{D}} m$ for every $\base{D} \supseteq \base{C}$. Hence, for every $\base{D} \supseteq \base{C}$, $\supp_{\base{D}} l$ implies $\supp_{\base{D}} m$. By~\ref{cl:CPL:inf}, this is exactly $l \supp_{\base{C}} m$, as required.
\end{proof}

\section{Completeness}
\label{sec:completeness}

Our aim is to establish the following:
\[
    \text{if } \Gamma \supp \phi, \text{ then } \Gamma \vdash_{\mathsf{NK}^\pm} \phi.
\]
We adapt the method Sandqvist~\cite{Sandqvist2015IL} used for intuitionistic propositional logic. The method transfers readily to other intuitionistic logics~\cite{BI,IMLL}, but for classical logic it has needed extra work and has not, so far, handled the disjunctive signs as primitive~\cite{Gheorghiu2024fol}. The bilateral setup makes it apply more directly. The strategy presupposes an unlimited supply of fresh contents: we assume a potentially infinite set $\mathcal{C}$, and hence a potentially infinite stock of propositions $l_1^+, l_1^-, l_2^+, l_2^-, l_3^+, l_3^-, \ldots$ --- \emph{potentially}, not \emph{actually}, infinite, in keeping with the anti-realist commitments of our metatheory.

\subsection{The Simulation Base}

The central construction is a canonical base $\base{N}$, called the \emph{simulation base}, which encodes the inferential definitions of all formulae in $\Gamma$ and $\phi$ and thereby bridges semantics and provability. Its construction proceeds as follows.

We begin by partitioning $\setLiterals := \setLiterals_1 \cup \setLiterals_2$ so that $\dual{(-)}: \setLiterals_1 \to \setLiterals_2$ is a bijection. Let $\textrm{Sub}(\Gamma, \phi)$ denote the set of subformulae of $\Gamma \cup \{\phi\}$, and let
\[
    \Xi := \textrm{Sub}(\Gamma, \phi)
\]
be the set of formulae whose inferential behaviour the simulation base must encode. We define $\kflat{(-)}$ to be the identity on $\Xi \cap \setLiterals$, and extend it to an injection
\begin{equation}
    \tag{$\ast$}
    \kflat{(-)}: \Xi \setminus \setLiterals \;\longrightarrow\; \setLiterals_1.
\end{equation}
Let $\ksharp{(-)}$ denote the left inverse of $\kflat{(-)}$. We extend it to the full $\setLiterals$ as follows:
\begin{itemize}
\item first, set $\ksharp{l} = l$ whenever $l$ is a literal not in the range of $\kflat{(-)}$;
    \item second, set $\ksharp{(\dual{l})} := \neg\,\ksharp{l}$ whenever $l$ is the $\kflat{(-)}$-image of a non-atomic formula. The bipartite structure of $\setLiterals$ ensures that for every non-atomic $\xi \in \Xi$,
\[
    \ksharp{\bigl(\dual{(\kflat{\xi})}\bigr)} = \neg\,\xi.
\]
\end{itemize}

Both operators extend pointwise to sets:
\[
    \kflat{\Delta} := \{\kflat{\delta} \mid \delta \in \Delta\}
    \qquad\text{and}\qquad
    \ksharp{L} := \{\ksharp{l} \mid l \in L\}.
\]
We restrict attention to the case in which the injection $(\ast)$ exists; this is automatic whenever $\Xi$ is finite, and in particular whenever $\Gamma$ is finite. Given this injection, $\base{N}$ is defined so that each rule $\rho$ of $\calculus{NK}^\pm$ is simulated by a corresponding rule $\kflat{\rho}$ in $\base{N}$. For instance, the conjunction rules of $\calculus{NK}^\pm$ are simulated in $\base{N}$ by the triple
\[
    \infer{r}{r_1 & r_2} \qquad \infer{r_1}{r} \qquad \infer{r_2}{r},
\]
where $r$, $r_1$, and $r_2$ are atoms formally associated with $\rho$, $\rho_1$, and $\rho_2$, respectively. The full set of rules constituting $\base{N}$ is displayed in Figure~\ref{fig:nksimulation}, where $\alpha$, $\beta$, $\gamma$ range over $\Xi$ (so that each rule is included only when the formulae it mentions lie in $\Xi$) and $l$ ranges over $\setLiterals$. Rules simulating $\exc$ and $\dil$ are unnecessary, since these are already built into the definition of derivability in a base (Definition~\ref{def:derivability-in-a-base}).

\begin{figure}[h]
    \hrule\vspace{2mm}
    \[
    \begin{array}{c}
        \infer[\kflat{\irn{\top}}]{~~\kflat{\top}~~}{}
        \qquad
        \infer[\kflat{\ern{\bot}}]{l}{\kflat{\bot}}
        \\[3mm]
        \infer[\kflat{\irn{\to}}]{\kflat{(\alpha \to \beta)}}{\deduce{\kflat{\beta}}{[\kflat{\alpha}]}}
        \qquad
        \infer[\kflat{\ern{\to}}]{\kflat{\beta}}{\kflat{(\alpha \to \beta)} & \kflat{\alpha}}
        \\[3mm]
        \infer[\kflat{\irn{\land}}]{\kflat{(\alpha \land \beta)}}{\kflat{\alpha} & \kflat{\beta}}
        \qquad
        \infer[\kflat{\ern{\land_1}}]{\kflat{\alpha}_1}{\kflat{(\alpha_1 \land \alpha_2)}}
        \quad
        \infer[\kflat{\ern{\land_2}}]{\kflat{\alpha}_2}{\kflat{(\alpha_1 \land \alpha_2)}}
        \\[3mm]
        \infer[\kflat{\irn{\lor_1}}]{\kflat{(\gamma_1 \lor \gamma_2)}}{\kflat{\gamma}_1}
        \quad
        \infer[\kflat{\irn{\lor_2}}]{\kflat{(\gamma_1 \lor \gamma_2)}}{\kflat{\gamma}_2}
        \qquad
        \infer[\kflat{\ern{\lor}}]{l}{\kflat{(\gamma_1 \lor \gamma_2)} & \deduce{l}{[\kflat{\gamma}_1]} & \deduce{l}{[\kflat{\gamma}_2]}}
    \end{array}
    \]
    \vspace{1mm}\hrule
    \caption{The simulation base $\base{N}$.}
    \label{fig:nksimulation}
\end{figure}

\subsection{The Completeness Argument}

Completeness follows from three lemmas in combination.

\begin{itemize}
    \item \textbf{AtComp.}\labelandtag{lem:CPL:basic-completeness}{AtComp} For any $L \subseteq \setLiterals$, any $l \in \setLiterals$, and any base $\base{B}$,
    \[
    L \supp_{\base{B}} l \qquad \text{iff} \qquad L \proves{\base{B}} l.
    \]
    \item \textbf{Flattening.}\labelandtag{lem:CPL:flat-equivalence}{Flattening} For any $\xi \in \Xi$ and any $\base{N}' \supseteq \base{N}$,
    \[
    \supp_{\base{N}'} \kflat{\xi} \qquad \text{iff} \qquad \supp_{\base{N}'} \xi.
    \]
    \item \textbf{Naturalizing.}\labelandtag{lem:CPL:sharpening}{Naturalizing} For any $L \subseteq \setLiterals$ and $l \in \setLiterals$,
    \[
    L \proves{\base{N}} l \qquad \text{implies} \qquad \ksharp{L} \vdash_{\mathsf{NK}^\pm} \ksharp{l}.
    \]
\end{itemize}

\ref{lem:CPL:basic-completeness} follows from Proposition~\ref{prop:CPL:atomic-cut}. \ref{lem:CPL:flat-equivalence} is established by the same argument as in Sandqvist~\cite{Sandqvist2015IL}; since that argument makes no essential use of the bilateral setup, we omit the details. \ref{lem:CPL:sharpening}, by contrast, requires particular care on account of the rules governing duality.

\begin{Lemma}[\ref{lem:CPL:sharpening}]
    If $L \proves{\base{N}} l$, then $\ksharp{L} \vdash_{\mathsf{NK}^\pm} \ksharp{l}$.
\end{Lemma}

\begin{proof}
    We proceed by induction on the derivation of $L \proves{\base{N}} l$ (Definition~\ref{def:derivability-in-a-base}).

    \begin{itemize}

        \item[\textsc{ref}.] From $l, L \proves{\base{N}} l$, the claim $\ksharp{l}, \ksharp{L} \vdash_{\mathsf{NK}^\pm} \ksharp{l}$ holds trivially, by the one-step derivation consisting of $\ksharp{l}$ alone.

        \item[\textsc{app}$_1$.] Every rule of the relevant form in $\base{N}$ is a flattened rule of $\calculus{NK}^\pm$. The conclusion follows from the induction hypothesis and the definition of $\ksharp{(-)}$.

        \item[\textsc{app}$_2$.] The rules of the relevant form in $\base{N}$ are again flattened rules of $\calculus{NK}^\pm$. The conclusion follows from the induction hypothesis and the definition of $\kflat{(-)}$.

        \item[\textsc{exc}.] Suppose for some $m\in \setLiterals$ we have $L \proves{\base{N}} m$ and $L \proves{\base{N}} \dual{m}$. We distinguish two cases.

        \smallskip
        \emph{Case 1: $m$ lies in the range of $\kflat{(-)}$.} That is, there exists a non-atomic $\xi \in \Xi$ such that $m=\kflat{\xi}$, so that $\ksharp{m} = \xi$ and $\ksharp{(\dual{m})} = \neg\,\xi$. By the induction hypothesis, 
        \[
            \ksharp{L} \vdash_{\mathsf{NK}^\pm} \xi
            \qquad\text{and}\qquad
            \ksharp{L} \vdash_{\mathsf{NK}^\pm} \neg\,\xi.
        \]
        Applying $\ern{\to}$ followed by $\ern{\bot}$ yields $\ksharp{L} \vdash_{\mathsf{NK}^\pm} \psi$ for every $\psi$. Choosing $\psi = \ksharp{l}$ completes the case. 

        \smallskip
        \emph{Case 2: $m$ is not the $\kflat{(-)}$-image of a non-atomic formula.} Then $m$ is either a literal of $\Xi$ or a literal outside the range of $\kflat{(-)}$; in both cases $\ksharp{m}=m$ and $\ksharp{(\dual{m})}=\dual{m}$. By the induction hypothesis,
        \[
            \ksharp{L} \vdash_{\mathsf{NK}^\pm} m
            \qquad\text{and}\qquad
            \ksharp{L} \vdash_{\mathsf{NK}^\pm} \dual{m}.
        \]
        Applying $\exc$ yields $\ksharp{L} \vdash_{\mathsf{NK}^\pm} l$, as required.

        \item[\textsc{dil}.] Suppose $L, m \proves{\base{N}} l$ and $L, \dual{m} \proves{\base{N}} l$, so that $L \proves{\base{N}} l$. By the induction hypothesis, $\ksharp{L}, \ksharp{m} \vdash_{\mathsf{NK}^\pm} \ksharp{l}$ and $\ksharp{L}, \ksharp{(\dual{m})} \vdash_{\mathsf{NK}^\pm} \ksharp{l}$. If $m$ is not the $\kflat{(-)}$-image of a non-atomic formula, then $\ksharp{m} = m$ and $\ksharp{(\dual{m})} = \dual{m}$, and the conclusion follows by a direct application of $\dil$ with case literal $m$. Otherwise $m = \kflat{\xi}$ for some non-atomic $\xi \in \Xi$, so the two hypotheses read $\ksharp{L}, \xi \vdash_{\mathsf{NK}^\pm} \ksharp{l}$ and $\ksharp{L}, \neg\xi \vdash_{\mathsf{NK}^\pm} \ksharp{l}$; the conclusion $\ksharp{L} \vdash_{\mathsf{NK}^\pm} \ksharp{l}$ then follows because $\xi \lor \neg\xi$ is derivable in $\mathsf{NK}^\pm$ (by the result of Negri and von Plato~\cite{Negri2008}, since $\dil$ gives excluded middle for literals and this lifts to all formulae) together with ${\lor}\mathsf{E}$.

    \end{itemize}

    This completes the induction.
\end{proof}

\begin{Theorem}[Completeness]
\label{thm:CPL:completeness-support}
    Suppose $\Gamma$ and $\phi$ admit an injection of the form $(\ast)$. If $\Gamma \supp \phi$, then $\Gamma \vdash_{\mathsf{NK}^\pm} \phi$.
\end{Theorem}

\begin{proof}
    Assume $\Gamma \supp \phi$. Let $\kflat{(-)}$ be a flattening operator of the form $(\ast)$ and $\base{N}$ its associated simulation base. By~\ref{lem:CPL:flat-equivalence}, $\kflat{\Gamma} \supp_{\base{N}} \kflat{\phi}$. By~\ref{lem:CPL:basic-completeness}, $\kflat{\Gamma} \proves{\base{N}} \kflat{\phi}$. By~\ref{lem:CPL:sharpening}, $\ksharp{(\kflat{\Gamma})} \vdash_{\mathsf{NK}^\pm} \ksharp{(\kflat{\phi})}$ --- that is, $\Gamma \vdash_{\mathsf{NK}^\pm} \phi$, as required.
\end{proof}
To see the three lemmas at work, we return to Peirce's law. Example~\ref{ex:SemanticPeircesLaw} showed that $\supp ((p\to q)\to p)\to p$ for atoms $p$ and $q$, so the completeness argument must recover an $\mathsf{NK}^\pm$-proof of that formula; we now trace how it does so. As before, we keep to the atomic case to avoid extraneous complexity.
\begin{Example}[Peirce's law, continued]\label{ex:SyntacticPeircesLaw}
    Fix a flattening operator $\kflat{(-)}$ of the form $(\ast)$. The subformulae of $((p\to q)\to p)\to p$ are
    \[
        \Xi=\{p,\ q,\ p\to q,\ (p\to q)\to p,\ ((p\to q)\to p)\to p\}.
    \]
    Since $\to$ is the only connective in $\Xi$, the simulation base $\base{N}$ consists of the following six rules:
    \[
        \begin{array}{c}
            \\[3mm]
            \infer[R_1]{\kflat{(p\to q)}}{\deduce{q}{[p]}} \qquad  \infer[R_2]{q}{p & \kflat{(p\to q)}}\\[3mm]
            \infer[R_3]{\kflat{((p\to q)\to p)}}{\deduce{p}{[\kflat{(p\to q)}]}} \qquad  \infer[R_4]{p}{\kflat{(p\to q)} & \kflat{((p\to q)\to p)}}\\[3mm]
            \infer[R_5]{\kflat{(((p\to q)\to p)\to p)}}{\deduce{p}{[\kflat{((p\to q)\to p)}]}} \qquad  \infer[R_6]{p}{\kflat{((p\to q)\to p)} & \kflat{(((p\to q)\to p)\to p)}}\\[3mm]
        \end{array}
    \]
    The odd-numbered rules simulate $\irn\to$, and the even-numbered rules simulate $\ern\to$.

    We now follow the proof of Theorem~\ref{thm:CPL:completeness-support}. Validity is support in every base (Definition~\ref{def:support}), so in particular $\supp_{\base{N}} ((p\to q)\to p)\to p$. By the~\ref{lem:CPL:flat-equivalence} lemma, $\supp_{\base{N}} \kflat{(((p\to q)\to p)\to p)}$; by the~\ref{lem:CPL:basic-completeness} lemma, $\proves{\base{N}} \kflat{(((p\to q)\to p)\to p)}$. Proof search in $\base{N}$ yields the following derivation:
    \[
        \infer[{R_5}^1]{\kflat{(((p\to q)\to p)\to p)}}{
            \infer[\text{\textsc{dil}}^2]{p}{
                [p]^2
                &
                \infer[R_4]{p}{
                    [\kflat{((p\to q)\to p)}]^1
                    &
                    \infer[{R_1}^3]{\kflat{(p\to q)}}{
                        \infer[\text{\textsc{exc}}]{q}{
                            [p]^3
                            &
                            [p^\bot]^2
                        }
                    }
                }
            }
        }
    \]
        The~\ref{lem:CPL:sharpening} lemma turns this derivation into an $\mathsf{NK}^\pm$-proof: each flattened atom is read back as the formula it stands for, $R_1$ and $R_5$ become instances of $\irn\to$, and $R_4$ becomes an instance of $\ern\to$. The result is the following proof of Peirce's law:
    \[
        \infer[{\irn\to}^1]{((p\to q)\to p)\to p}{
            \infer[{\dil}^2]{p}{
                [p]^2
                &
                \infer[\ern\to]{p}{
                    [(p\to q)\to p]^1
                    &
                    \infer[{\irn\to}^3]{p\to q}{
                        \infer[\exc]{q}{
                            [p]^3
                            &
                            [p^\bot]^2
                        }
                    }
                }
            }
        }
    \]
\end{Example}

\section{Conclusion} \label{sec:conclusion}

We have given a \textit{proof-theoretic semantics} (P-tS) for classical propositional logic. Unlike the existing P-tS for classical logic of Sandqvist~\cite{Sandqvist2015hypothesis,Sandqvist2009CL}, Makinson~\cite{makinson2014inferential}, and Gheorghiu~\cite{Gheorghiu2024fol}, ours is built on \emph{literals}: dual literals are the assertion and the denial of the same content, and the bilateralist distinction stops there, since only literals bear illocutionary force and complex formulae are logico-grammatical structures of reasoning. Duality at the atomic level is governed by two inferential principles --- exclusion ($\exc$) and case analysis on literals ($\dil$) --- which, by a result of Negri and von Plato~\cite{Negri2008}, suffice to recover classical logic from an intuitionistic base. The treatment of classical logic thus follows that of \emph{intuitionistic} logic step for step. That is the advantage: intuitionistic logics have received a far more systematic proof-theoretic semantics --- see Sandqvist~\cite{Sandqvist2015IL}, Pym et al.~\cite{IMLL,BI,NAF,PTV-BES}, and Buzoku~\cite{Buzoku2024} --- and this paper opens the way to a treatment of classical logics that is systematic in the same manner.

Systematic proof-theoretic semantics for classical logics do already exist. Eckhardt and Pym~\cite{Eckhardt,Eckhardt2} build directly on Makinson~\cite{makinson2014inferential} to give a proof-theoretic semantics for the normal modal logics. That work is unarguably proof-theoretic semantics in the sense of this paper, but it stays close to the model-theoretic semantics of the logics concerned. We hope that the approach presented here offers a genuinely \emph{alternative} semantics. The hope has some ground: for intuitionistic propositional logic, the relationship between the P-tS of Sandqvist~\cite{Sandqvist2015IL} and the model-theoretic semantics of Kripke~\cite{Kripke1965} is still open.

\bibliographystyle{asl}
\bibliography{bib}

\end{document}